\documentclass{amsart}

\usepackage[utf8]{inputenc} 
\usepackage[T1]{fontenc} 
\usepackage{amsmath,amssymb,amsthm} 
\usepackage{hyperref} 
\usepackage{enumitem} 
\usepackage{array} 
\usepackage{mathrsfs} 
\usepackage{hyphenat} 
\usepackage[longend,linesnumbered]{algorithm2e} 
\usepackage{tikz-cd} 
\usepackage{tikz} 

\SetAlgoCaptionSeparator{.}
\SetAlCapFnt{\sc}

\tikzset{every picture/.append style={>=latex}}
\tikzset{every node/.append style={fill=white,font=\footnotesize,inner sep=2pt}}
\tikzset{every loop/.style={min distance=10pt,looseness=10,in=70,out=110}}

\newcommand{\nn}{\mathbb{N}} 
\newcommand{\zz}{\mathbb{Z}} 
\newcommand{\isom}{\cong} 
\newcommand{\free}{F} 
\newcommand{\freepro}{\widehat{F}} 
\newcommand{\emptyw}{\varepsilon} 
\newcommand{\ret}[1]{\mathcal{R}_{#1}} 
\newcommand{\green}[1]{\mathcal{#1}} 
\newcommand{\retsubs}[2]{#1_{#2}} 
\newcommand{\loo}[1]{\theta_{#1}} 
\newcommand{\var}[1]{\mathbf{#1}} 
\newcommand{\rest}[2]{r_{#2}({#1})} 
\newcommand{\automaton}[1]{\mathscr{#1}} 
\newcommand{\given}{ : } 
\newcommand{\from}{\colon} 

\DeclareMathOperator{\Card}{Card} 
\DeclareMathOperator{\img}{Im} 
\DeclareMathOperator{\End}{End} 
\DeclareMathOperator{\sol}{sol}
\DeclareMathOperator{\nil}{nil}
\DeclareMathOperator{\cl}{Cl}
\DeclareMathOperator{\lcm}{lcm}

\theoremstyle{plain}
\newtheorem{theorem}{Theorem}[section]
\newtheorem{lemma}[theorem]{Lemma}
\newtheorem{proposition}[theorem]{Proposition}
\newtheorem{corollary}[theorem]{Corollary}

\theoremstyle{definition}
\newtheorem{definition}[theorem]{Definition}

\theoremstyle{remark} 
\newtheorem{remark}[theorem]{Remark}
\newtheorem{example}[theorem]{Example}

\begin{document}

\title{Freeness of Schützenberger groups of primitive substitutions}  

\thanks{This work was conducted with the support of the Centre for Mathematics of the University of Coimbra (UIDB/00324/2021, funded by the Portuguese Government through FCT/MCTES) and the Centre for Mathematics of the University of Porto (UIDB/00144/2020, funded by the Portuguese Government through FCT/MCTES). The author also acknowledges the financial support of the Portuguese Government through a PhD grant (PD/BD/150350/2019)}

\author[H. Goulet-Ouellet]{Herman Goulet-Ouellet}

\address{University of Coimbra, CMUC, Department of Mathematics, Apartado 3008, EC Santa Cruz, 3001-501 Coimbra, Portugal}

\email{hgouletouellet@student.uc.pt}



\begin{abstract}
    Our main goal is to study the freeness of Schützenberger groups defined by primitive substitutions. Our findings include a simple freeness test for these groups, which is applied to exhibit a primitive invertible substitution with corresponding non-free Schützenberger group. This constitutes a counterexample to a result of Almeida dating back to 2005. We also give some early results concerning relative freeness of Schützenberger groups, a question which remains largely unexplored.
\end{abstract}

\maketitle

\section{Introduction}
\label{s:intro}

In \cite{Almeida2007}, Almeida unveiled a connection between symbolic dynamical systems, or shift spaces, and maximal subgroups of free profinite monoids. More precisely, he proved that the topological closure inside the free profinite monoid of the language of a minimal shift space contains a unique regular $\green{J}$-class. By standard results from semigroup theory, all the maximal subgroups contained in a regular $\green{J}$-class define the same group up to isomorphism, known as its \emph{Schützenberger group}. In a profinite monoid, the Schützenberger group of a regular $\green{J}$-class is a profinite group. Thus, Almeida's correspondence associates to each minimal shift space a profinite group, and this defines a conjugacy invariant \cite{Costa2006}. 

In the study of Schützenberger groups corresponding to minimal shift spaces, the freeness question has been a recurring theme \cite{Almeida2007,Almeida2013,Almeida2016,Costa2011}. These groups are known to be free for the family of dendric shift spaces, also known as tree sets \cite[Theorem~6.5]{Almeida2016}. Notably, these include Arnoux-Rauzy shift spaces \cite[Example~3.2]{Berthe2015}, as well as shift spaces defined by regular interval exchange \cite[Theorem~4.3]{Berthe2015b}. On the other hand, failure of freeness was also observed, for instance in the shift space defined by the Thue--Morse substitution \cite[Theorem~7.6]{Almeida2013}. This raises the general question: when is the Schützenberger group defined by a minimal shift space free? At time of writing, this question remains largely open. A partial answer was proposed early on by Almeida, which argued that the Schützenberger group of a primitive invertible substitution must be free \cite[Corollary~5.7]{Almeida2007}. However, upon closer inspection, we noticed some gaps in the proof. This prompted us to investigate more closely the freeness question for Schützenberger groups of primitive substitutions, with an eye on the specific case of invertible substitutions. This paper aims to present the results of this investigation, which include a counterexample to \cite[Corollary~5.7]{Almeida2007}.

The paper is organized as follows. In Section~\ref{s:preliminaries}, we review some relevant background. In Section~\ref{s:presentations}, we discuss the notion of $\omega$\hyp{}presentation (a type of profinite presentation introduced in \cite{Almeida2013}) and we give a number of technical results. In Section~\ref{s:freeness}, we examine the link between freeness and $\omega$\hyp{}presentations. The main result of this section, Theorem~\ref{t:det}, provides a simple test for freeness of Schützenberger groups of primitive substitutions. Several examples are presented for which the test can be succesfully applied. In Section~\ref{s:relfree}, we study the Schützenberger groups of relatively invertible primitive substitutions, and more precisely the pseudovarieties generated by the finite quotients of such Schützenberger groups. The main result of this section has two consequences that are of particular interest to us. First, if a primitive substitution is invertible, then its Schützenberger group is relatively free if and only if it is absolutely free. Second, if a primitive substitution is unimodular and its Schützenberger group is relatively free, then it must be free with respect to a pseudovariety containing at least all finite nilpotent groups. Finally, Section~\ref{s:invnonfree} presents our counterexample to \cite[Corollary~4.7]{Almeida2007}, which consists of a primitive invertible substitution whose Schützenberger group is not free, and in fact not relatively free by the results of Section~\ref{s:relfree}.

\section{Preliminaries}
\label{s:preliminaries}

This section aims to provide some context and present most of the relevant background. Additional notions will be introduced in the course of the paper as they are needed. The monograph \cite{Almeida2020a} contains an in-depth treatment of most of the material we need. Here is a list of more specialized documents that may also be useful: on profinite groups and profinite presentations, \cite{Lubotzky2001,Ribes2010a}; on profinite semigroups and Schützenberger groups of primitive substitutions, \cite{Almeida2005,Almeida2007, Almeida2013}; on return sets and return substitutions, \cite{Durand1998,Durand1999}.

By an \emph{alphabet}, we mean a finite set $A$ whose elements are called \emph{letters}. We use $A_n$ as a shorthand for the alphabet $\{0,\dots,n-1\}$, $n\in\nn$. Let $\free(A)$ be the free group on the alphabet $A$ and $\freepro(A)$ be the free profinite group on $A$. We use the notation $\emptyw$ to denote the identity element of both $\free(A)$ and $\freepro(A)$, as well as the empty word. We denote by $\End(\free(A))$ the set of endomorphisms of $\free(A)$, and by $\End(\freepro(A))$ the set of continuous endomorphisms of $\freepro(A)$. An endomorphism $\phi\in\End(\free(A))$ admits a unique continuous extension $\widehat{\phi}\in\End(\freepro(A))$, called the \emph{profinite extension} of $\phi$. 

In this paper, we deal with profinite presentations in the sense of \cite{Lubotzky2001}. Formally, a \emph{presentation} of a profinite group $G$ is a pair formed by a set $A$ and a subset $R\subseteq\freepro(A)$ such that $G\isom\freepro(A)/N$, where $N$ is the closed normal subgroup of $\freepro(A)$ generated by $R$. We call $A$ the set of \emph{generators} and $R$ the set of \emph{relations}. The notation $G\isom\langle  A \mid R \rangle$ means that $(A,R)$ is a presentation of $G$. The minimal number of generators in a presentation of $G$ is denoted $d(G)$. A presentation realizing this minimum is called a \emph{minimal presentation}.

A \emph{substitution} is an endomorphism $\varphi$ of the free monoid $A^*$ over an alphabet $A$. Assuming $A$ has at least two letters, we say that $\varphi$ is \emph{primitive} if there exists $n\in\nn$ such that $b$ occurs in $\varphi^n(a)$, for all $a,b\in A$. On the other hand, if $A$ is a one-letter alphabet, then we say that $\varphi$ is \emph{primitive} if $\varphi(a) = a^n$ with $n>1$. A substitution $\varphi\from A^*\to A^*$ is called \emph{invertible} if its natural extension to an endomorphism of $\free(A)$ is an automorphism. Note that, if $A$ is a singleton, the only invertible substitution is the identity mapping, which is not primitive according to our definition. Thus, a primitive invertible substitution is always defined on at least two letters.

Following \cite[Section~5.5]{Almeida2020a}, a primitive substitution $\varphi\from A^*\to A^*$ defines a minimal shift space $X(\varphi)\subseteq A^\zz$. The language of this shift space, which we denote $L(\varphi)$, is the subset of $A^*$ formed by the factors of the words $\varphi^n(a)$ for all $n\in\nn$ and $a\in A$. Minimality of $X(\varphi)$ means that $L(\varphi)$ must be \emph{uniformly recurrent}. That is, $L(\varphi)$ is infinite, closed under taking factors, and satisfies the \emph{bounded gap property}: for all $u\in L(\varphi)$, there exists $n\in\nn$ such that $u$ is a factor of every word $w\in L(\varphi)$ with $|w|\geq n$. We say that $\varphi$ is \emph{periodic} if $X(\varphi)$ is a periodic shift space, or equivalently if $L(\varphi)$ is the language of factors in the powers of a given word $w\in A^+$. Otherwise, we say that $\varphi$ is \emph{aperiodic}.

Let $\widehat{A^*}$ be the free profinite monoid over an alphabet $A$. A result of Almeida shows that if $L\subseteq A^*$ is uniformly recurrent, then $\overline{L}\setminus A^*$ is a $\green{J}$-maximal regular $\green{J}$-class of $\widehat{A^*}$, where $\overline{L}$ is the topological closure of $L$ in $\widehat{A^*}$ \cite[Propositon~5.6.14]{Almeida2020a}. This in fact gives a bijective correspondence between uniformly recurrent languages (and thus minimal shift spaces) and $\green{J}$-maximal regular $\green{J}$-classes of $\widehat{A^*}$ \cite[Proposition~5.6.12]{Almeida2020a}. Standard results from semigroup theory imply that the maximal subgroups contained in $\overline{L}\setminus A^*$ are all isomorphic to the same profinite group, which is called the \emph{Schützenberger group} of the $\green{J}$-class (see for instance \cite[Section~3.6]{Almeida2020a}). In case $L = L(\varphi)$ is the language of a primitive substitution, we denote this group by $G(\varphi)$ and we call it the \emph{Schützenberger group of $\varphi$}. Note that if $\varphi$ is periodic, then $G(\varphi)$ is a free profinite group of rank 1 \cite[Exercise~5.20]{Almeida2020a}, so from now on we focus on the aperiodic case.

Two-sided return substitutions, introduced in \cite{Durand1999}, play an important role in the study of Schützenberger groups of primitive substitutions \cite{Almeida2013}. It is based on the notion of return words, which we recall now. Let $\varphi$ be a primitive substitution and $u,v\in L(\varphi)$ be such that $uv\in L(\varphi)$. By a \emph{return word} to $(u,v)$ in $L(\varphi)$, we mean a word $r\in A^*$ that separates two consecutive occurrences of $(u, v)$ in $L(\varphi)$. More precisely, it is a word $r\in A^*$ such that $urv$ is in $L(\varphi)$, starts and ends with $uv$, and contains exactly two occurrences of $uv$. The set of such words is denoted $\ret{u,v}$, and we call this a \emph{return set} of $\varphi$. For primitive substitutions, the return sets are always finite and non-empty (by uniform recurrence of $L(\varphi)$, see \cite[Proposition~4.2]{Berthe2015a}). Moreover, they generate free submonoids of $A^*$, for which they form bases. In other words, the return sets of primitive substitutions are \emph{codes} \cite[Lemma~17]{Durand1999}. A further property worth mentionning is that a primitive substitution is periodic if and only if one of its return sets is a singleton, if and only if all but finitely many of its return sets are singletons (see \cite[Proposition~2.8]{Durand1998} and \cite[Proposition~4.4]{Berthe2015a}).

By a \emph{connection}\footnote{The term \emph{connection} was coined by Almeida in \cite{Almeida2007}, where it was used under the condition $|u|=|v|=1$.} of a primitive substitution $\varphi$, we mean a pair of non-empty words $(u,v)$ such that $uv\in L(\varphi)$ and, for some positive integer $l$, $\varphi^l(u)$ ends with $u$ and $\varphi^l(v)$ starts with $v$. The least positive integer $l$ with that property is called the \emph{order} of the connection. If $(u,v)$ is a connection of $\varphi$ of order $k$, then $\varphi^k$ restricts to a primitive substitution of the free submonoid generated by $\ret{u,v}$ \cite[Lemma~21]{Durand1999}. This substitution, which we denote $\retsubs{\varphi}{u,v}$, is said to be a \emph{return substitution} of $\varphi$. All primitive substitutions have at least one connection, hence at least one return substitution \cite[Proposition~5.5.10]{Almeida2020a}.

Following the convention used in \cite{Durand1998}, we relabel return substitutions using the natural ordering of return words induced by leftmost occurrences. This ordering may be defined as follows. Let $(u,v)$ be a connection of order $k$ of a primitive substitution $\varphi$. First, by uniform recurrence of $L(\varphi)$, there exists $n\in\nn$ such that the word $u\varphi^{nk}(v)$ contains every word of the form $urv$, $r\in\ret{u,v}$. For $r,s\in\ret{u,v}$, we say that \emph{$r$ precedes $s$ in the leftmost occurrence ordering} if the leftmost occurrence of $urv$ in $u\varphi^{nk}(v)$ is located to the left of every occurrence of $usv$. Because $u\varphi^{nk}(v)$ is a prefix of $u\varphi^{mk}(v)$ whenever $m\geq n$, this ordering is independent of $n$. We view this as a bijection
\begin{equation*}
    \loo{u,v}\from A_{u,v}\to\ret{u,v}, \qquad \text{ where } A_{u,v} = \{0,\dots,\Card(\ret{u,v})-1\}.
\end{equation*}
The return substitution $\retsubs{\varphi}{u,v}$ can be defined as the unique substitution of $A_{u,v}^*$ satisfying the relation
\begin{equation*}
    \loo{u,v}\circ\retsubs{\varphi}{u,v} = \varphi^k\circ\loo{u,v},
\end{equation*}
where $\loo{u,v}$ is extended to an homomorphism $\loo{u,v}\from A_{u,v}^*\to A^*$.

\section{\texorpdfstring{$\omega$}{Omega}-presentations}
\label{s:presentations}

We recall that the continuous endomorphisms of a finitely generated profinite group form a profinite monoid (see for instance \cite[Section~3.12]{Almeida2020a})\footnote{For historical context, Hunter proved in \cite{Hunter1983} that the monoid of continuous endomorphisms of a finitely generated profinite semigroup is profinite for the compact-open topology. This result was rediscovered by Almeida \cite{Almeida2005} and generalized by Steinberg \cite{Steinberg2010}.}. This implies that for every such continuous endomorphism $\psi$, the closure of $\{\psi^{n} \given n\in\nn\}$ contains a unique idempotent element, which is denoted $\psi^\omega$. More information about $\omega$\hyp{}powers, including basic properties, can be found in \cite[Section~3.7]{Almeida2020a}. We now give the eponymous definition of this section.

\begin{definition}\label{d:omega}
    Let $G$ be a profinite group. An \emph{$\omega$\hyp{}presentation} of $G$ is a profinite presentation of the form
    \begin{equation*}
        G \isom \langle A \mid \widehat{\phi}^\omega(a)a^{-1} \given a\in A \rangle,
    \end{equation*}
    where $A$ is a finite set and $\phi\in\End(\free(A))$. We then say that $\phi$ \emph{defines an $\omega$\hyp{}presentation} of $G$.
\end{definition}

The number of generators of an $\omega$\hyp{}presentation of $G$ defined by an endomorphism of $\free(A)$ is equal to $\Card(A)$. Hence, such an $\omega$\hyp{}presentation is minimal as a presentation of $G$ precisely when $\Card(A) = d(G)$. We call this a \emph{minimal $\omega$\hyp{}presentation}. We also note that the following alternative notation is sometimes used for $\omega$\hyp{}presentations, using relations instead of relators, for instance in \cite{Almeida2013}:
\begin{equation*}
    G \isom \langle A \mid \widehat{\phi}^\omega(a)=a \ (a\in A) \rangle.
\end{equation*}

The next lemma gives a different way to interpret $\omega$\hyp{}presentations. We use essentially the same argument as \cite[Proposition~1.1]{Lubotzky2001}, where it was attributed to Kovács. \begin{lemma}\label{l:kovacs}
    If $\phi\in\End(\free(A))$ defines an $\omega$\hyp{}presentation of a profinite group $G$, then $G\isom\img(\widehat{\phi}^\omega)$.
\end{lemma}

\begin{proof}
    Let $\psi=\widehat{\phi}$. It suffices to show that the closed normal subgroup $K$ of $\freepro(A)$ generated by $\{ \psi^\omega(a)a^{-1} \given a \in A\}$ is equal to $\ker(\psi^\omega)$. Since $\psi^\omega$ is an idempotent endomorphism, the following equalities hold:
    \begin{equation*}
        \psi^\omega(\psi^\omega(a)a^{-1}) = \psi^\omega(\psi^\omega(a))\psi^{\omega}(a^{-1}) = \psi^\omega(a)\psi^{\omega}(a^{-1}) = \psi^\omega(aa^{-1}) = \emptyw.
    \end{equation*}
    Therefore, $K$ is contained in $\ker(\psi^\omega)$. 

    To prove the reverse inclusion, we show that $K$ contains every element of the form $\psi^\omega(x)x^{-1}$ with $x\in\freepro(A)$. The desired inclusion clearly follows since $x\in\ker(\psi^\omega)$ implies $x^{-1} = \psi^\omega(x)x^{-1}$. Consider the following subset of $\freepro(A)$:
    \begin{equation*}
        \{ x\in\freepro(A) \given \psi^\omega(x)x^{-1}\in K \}.
    \end{equation*}
    Routine arguments show that this set forms a closed subgroup of $\freepro(A)$ which contains $A$. Hence, it must be equal to $\freepro(A)$, and this finishes the proof.
\end{proof}

Our motivation for introducing $\omega$\hyp{}presentations is a key result due to Almeida and Costa, which is stated below. It allows to effectively compute an $\omega$\hyp{}presentation for the Schützenberger group of every primitive aperiodic substitution, and will serve as our starting point in Section~\ref{s:invnonfree}. The original statement is restricted to connections $(u,v)$ satisfying $|u|=|v|=1$, but the proof works as long as $u, v\neq\emptyw$. 
\begin{theorem}[{\cite[Theorem~6.2]{Almeida2013}}]\label{t:return}
    Let $\varphi$ be a primitive aperiodic substitution and $(u,v)$ be a connection of $\varphi$. Then $G(\varphi)$ has the following $\omega$-presentation:
    \begin{equation*}
        G(\varphi)\isom \langle A_{u,v} \mid \widehat{\retsubs{\varphi}{u,v}}^\omega(a)a^{-1} \given a \in A_{u,v} \rangle.
    \end{equation*}
\end{theorem}
 
In other words, every return substitution of $\varphi$ defines an $\omega$-presentation of $G(\varphi)$.

\begin{remark}\label{r:proper}
    Assume that $\varphi$ is also \emph{proper}, meaning that there are $a_1,a_2\in A$ and $n\in\nn$ such that $\varphi^n(b)\in a_1A^*\cup A^*a_2$ for all $b\in A$. Then $G(\varphi)$ has the more straightforward $\omega$\hyp{}presentation $G(\varphi)\isom\langle A \mid \widehat{\varphi}^\omega(a)a^{-1} \given a\in A \rangle$ \cite[Theorem~6.4]{Almeida2013}. That is to say, $\varphi$ defines an $\omega$-presentation of its own Schützenberger group. Further noting that return substitutions are always proper \cite[Lemma~21]{Durand1999}, it follows that a return substitution $\retsubs{\varphi}{u,v}$ defines an $\omega$-presentation of both $G(\varphi)$ and $G(\retsubs{\varphi}{u,v})$. Hence, the two Schützenberger groups are isomorphic.
\end{remark}

\begin{example}\label{e:morse1}
    The \emph{Thue--Morse substitution} is the binary substitution $\tau$ defined by
    \begin{equation*}
        \begin{array}{rlll}
            \tau\from & 0 & \mapsto & 01\\
                      & 1 & \mapsto & 10.
        \end{array}
    \end{equation*}
    This substitution is clearly primitive and it is well known to be aperiodic. Moreover, it is easily verified that the pairs $(0,1)$, $(0,10)$ are connections of $\tau$ of order 2. Computing the corresponding return substitutions (for instance using the algorithm described in Section~\ref{s:invnonfree}), one obtains the following substitutions, both defined on the alphabet $A_4 = \{0,1,2,3\}$:
    \begin{equation*}
        \begin{array}{rlll}
            \retsubs{\tau}{0,1}\from & 0 & \mapsto & 0123\\
                            & 1 & \mapsto & 013\\
                            & 2 & \mapsto & 02123\\
                            & 3 & \mapsto & 0213
        \end{array}\quad
        \begin{array}{rlll}
            \retsubs{\tau}{0,10}\from & 0 &\mapsto & 01\\
                             & 1 &\mapsto & 023132\\
                             & 2 &\mapsto & 0232\\
                             & 3 &\mapsto & 0131.
        \end{array}
    \end{equation*}

    By Theorem~\ref{t:return}, the substitutions $\retsubs{\tau}{0,1}$ and $\retsubs{\tau}{0,10}$ define $\omega$\hyp{}presentations of the Schützenberger group $G(\tau)$, which means
    \begin{equation*}
        G(\tau) \isom \langle A_4 \mid \widehat{\retsubs{\tau}{0,1}}^\omega(a)a^{-1} \given a\in A_4\rangle\isom \langle A_4 \mid \widehat{\retsubs{\tau}{0,10}}^\omega(a)a^{-1} \given a\in A_4\rangle.
    \end{equation*}
\end{example}

As it was observed in \cite{Almeida2013}, the $\omega$\hyp{}presentations given by return substitutions (or indeed by the substitution itself in the proper case) are not always minimal. We now introduce a simple method for reducing the number of generators in $\omega$\hyp{}presentations. Let $\phi$ be an element of $\End(\free(A))$. We denote by $\rest{\phi}{n}$ the restriction of $\phi$ to an endomorphism of $\img(\phi^n)$.
\begin{proposition}\label{p:restriction}
    If $\phi\in\End(\free(A))$ defines an $\omega$\hyp{}presentation of a profinite group $G$, then, for every non-negative integer $n$, the endomorphism $\rest{\phi}{n}$ defines an $\omega$\hyp{}presentation of $G$ with at most $\Card(A)$ generators.
\end{proposition}

\begin{proof}
    By the Nielsen--Schreier theorem, $\img(\phi^n)=\free(B)$ for some finite set $B$. Moreover, since $\free(B)$ is generated by $\phi^n(A)$, we have $\Card(B)\leq\Card(A)$. It remains to show that $\rest{\phi}{n}$ defines an $\omega$-presentation of $G$, which by Lemma~\ref{l:kovacs} amounts to showing that $\img(\widehat{\rest{\phi}{n}}^\omega) \isom \img(\widehat{\phi}^\omega)$. 

    Let $\eta\from\free(B)\to\free(A)$ be the homomorphism induced by the inclusion $\free(B)\subseteq\free(A)$ and let $\widehat{\eta}\from\freepro(B)\to\freepro(A)$ be its profinite extension. Since $\eta$ is injective, so is $\widehat{\eta}$ by \cite[Theorem~4.6.7]{Almeida2020a}. Moreover, from the equality $\phi\circ\eta=\eta\circ\rest{\phi}{n}$, we deduce that the following diagram is commutative:
    \begin{center}
        \begin{tikzcd}
            \freepro(B) \rar{\widehat{\eta}} \dar{\widehat{\rest{\phi}{n}}} & \freepro(A) \dar{\widehat{\phi}} \\
            \freepro(B) \rar{\widehat{\eta}} & \freepro(A).
        \end{tikzcd}
    \end{center}
    Hence, $\widehat{\eta}$ restricts to a continuous isomorphism $\img(\widehat{\rest{\phi}{n}}^\omega) \isom \widehat{\phi}^\omega(\img(\widehat{\eta}))$. Noting the equalities 
    \begin{equation*}
        \img(\widehat{\eta}) = \overline{\img(\eta)} = \overline{\img(\phi^n)} = \img(\widehat{\phi}^n),
    \end{equation*}
    it then suffices to show that $\widehat{\phi}^\omega(\img(\widehat{\phi}^n)) = \img(\widehat{\phi}^\omega)$. And indeed, we have
    \begin{equation*}
        \widehat{\phi}^\omega(\img(\widehat{\phi}^n)) \subseteq \img(\widehat{\phi}^\omega) = \widehat{\phi}^\omega(\img(\widehat{\phi}^\omega)) \subseteq \widehat{\phi}^\omega(\img(\widehat{\phi}^n)).\qedhere
    \end{equation*}
\end{proof}

Note that the restriction operation satisfies $\rest{\rest{\phi}{n}}{k} = \rest{\phi}{n+k}$. Therefore, $\{\rest{\phi}{n}\}_{n\in\nn}$ gives a sequence of $\omega$\hyp{}presentations of the same profinite group with weakly decreasing numbers of generators. The next result tells us exactly when the number of generators stabilizes. The proof mostly boils down to the well-known fact that free groups of finite rank enjoy the \emph{Hopfian property}, which can be stated as follows: every surjective homomorphism between two free groups of the same finite rank is an isomorphism. See for instance \cite[Theorem~41.52]{Neumann1967}.
\begin{proposition}
    Let $\phi$ define an $\omega$\hyp{}presentation of a profinite group $G$. For every two non-negative integers $m, n\in\nn$ with $n<m$, the $\omega$\hyp{}presentations of $G$ defined by $\rest{\phi}{n}$ and $\rest{\phi}{m}$ have the same number of generators if and only if $\rest{\phi}{n}$ is injective.
\end{proposition}

\begin{proof}
    We start by noting that $\rest{\phi}{n}^{m-n}$ is a continuous surjective homomorphism from $\img(\phi^n)$ to $\img(\phi^m)$. If $\rest{\phi}{n}$ and $\rest{\phi}{m}$ define $\omega$\hyp{}presentations with the same number of generators, then $\img(\phi^n)$ and $\img(\phi^m)$ are free groups of the same rank, and by the Hopfian property, $\rest{\phi}{n}^{m-n}$ is an isomorphism. In particular, $\rest{\phi}{n}^{m-n}$ is injective, and since $m-n\geq 1$ so is $\rest{\phi}{n}$. 

    Conversely, if $\rest{\phi}{n}$ is injective, then $\rest{\phi}{n}^{m-n}\from\img(\phi^n)\to\img(\phi^m)$ is an isomorphism. Thus, $\img(\phi^n)$ and $\img(\phi^m)$ are free groups of the same rank and the $\omega$\hyp{}presentations defined by $\rest{\phi}{m}$ and $\rest{\phi}{n}$ have the same number of generators.
\end{proof}

We immediately deduce the following.
\begin{corollary}\label{c:injective}
    Let $\phi$ define an $\omega$\hyp{}presentation of a profinite group $G$. If $\phi$ is not injective, then there exists an $\omega$\hyp{}presentation of $G$ with strictly less generators. In particular, if $\phi$ defines a minimal $\omega$\hyp{}presentation of $G$, then $\phi$ must be injective.
\end{corollary}

The following example shows that injective endomorphisms can also define non-minimal $\omega$-presentations.

\begin{example}[Continued from Example~\ref{e:morse1}]\label{e:morse2}
    One can show that the endomorphism of $\free(A_4)$ induced by the return substitution $\retsubs{\tau}{0,1}$ is not injective. For instance,
    \begin{equation*}
        02^{-1}02^{-1}31^{-1}20^{-1} \in \ker(\retsubs{\tau}{0,1}).
    \end{equation*}
    Hence, by Corollary~\ref{c:injective}, the $\omega$\hyp{}presentation defined by $\retsubs{\tau}{0,1}$ is not minimal.

    On the other hand, $\retsubs{\tau}{0,10}$ extends to an injective endomorphism of $\free(A_4)$. One way to see this is to show that the set $\{03^{-1}, 31^{-1}, 3232, 2^{-1}12^{-1}3^{-1}\}$ is a basis of $\img(\retsubs{\tau}{0,10})$. More precisely, it is the basis determined (as in \cite[Lemma~6.1]{Kapovich2002}) by the spanning tree of the Stallings automaton of $\img(\retsubs{\tau}{0,10})$ given in Figure~\ref{f:stallings-morse}.
    \begin{figure}\centering
        \begin{tikzpicture}[scale=.6]
            \node (s0) at (23.799bp,64.799bp) [draw,circle, double] {$s_0$};
            \node (s1) at (108.4bp,109.8bp) [draw,circle] {$s_1$};
            \node (s2) at (108.4bp,19.799bp) [draw,circle] {$s_2$};
            \node (s3) at (192.99bp,64.799bp) [draw,circle] {$s_3$};
            \draw [->] (s0) to[bend left] node {$0$} (s1);
            \draw [->,dashed] (s0) to node {$3$} (s1);
            \draw [<-] (s0) to[bend right] node {$1$} (s1);
            \draw [<-,dashed] (s0) to node {$2$} (s2);
            \draw [->,dashed] (s1) to node {$2$} (s3);
            \draw [->] (s2) to[bend left] node {$1$} (s3);
            \draw [<-] (s2) to[bend right] node {$3$} (s3);
        \end{tikzpicture}
        \caption{Stallings automaton of the image of the return substitution $\retsubs{\tau}{0,10}$ viewed as an endomorphism of $\free(A_4)$. The distinguished state is identified by a double circle and the dashed edges form a spanning tree.}
        \label{f:stallings-morse}
    \end{figure}
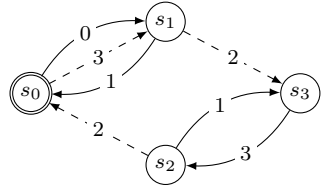
    Even though $\retsubs{\tau}{0,10}$ is injective, it does not define a minimal $\omega$\hyp{}presentation of $G(\tau)$, since it has the same number of generators as the non-minimal $\omega$\hyp{}presentation defined by $\retsubs{\tau}{0,1}$.

    According to Proposition~\ref{p:restriction}, a shorter $\omega$\hyp{}presentation of $G(\tau)$ is defined by the restriction $\rest{\retsubs{\tau}{0,1}}{1}$. Here is the endomorphism $\rest{\retsubs{\tau}{0,1}}{1}$ expressed in the basis $\{ 3^{-1}2,\ 20^{-1},\ 2^{-1}302^{-1}1 \}$ of $\img(\retsubs{\tau}{0,1})$: 
    \begin{equation*}
        \begin{array}{rlll}
            \rest{\retsubs{\tau}{0,1}}{1}\from & 0 & \mapsto & 02110\\
                               & 1 & \mapsto & 10021\\
                               & 2 & \mapsto & 2.
        \end{array}
    \end{equation*}

    Another $\omega$\hyp{}presentation of $G(\tau)$ with 3 generators was obtained, by other means, in \cite{Almeida2013}, where it is also shown that $d(G(\tau))=3$ \cite[Theorem~7.7]{Almeida2013}. Therefore, the $\omega$\hyp{}presentation defined by $\rest{\retsubs{\tau}{0,1}}{1}$ is minimal.
\end{example}

\section{Freeness via \texorpdfstring{$\omega$}{Omega}-presentations}
\label{s:freeness}

In this section, we present a few key results concerning freeness of profinite groups with $\omega$\hyp{}presentations. Given an endomorphism $\phi$ of $\free(A)$, the \emph{incidence matrix} of $\phi$ is the matrix $M(\phi)\in\zz^{A\times A}$ defined by $M(\phi)_{a,b} = |\phi(a)|_b$, where $|-|_b\from\free(A)\to\zz$ is the unique group homomorphism extending the Kronecker delta function $\delta_b\from A\to\zz$. The main result of the section is the following theorem, which provides a simple freeness test for Schützenberger groups of primitive substitutions. We will make use of this test in Section~\ref{s:invnonfree} to exhibit a primitive invertible substitution whose Schützenberger group is not free. Two examples where this test can be applied are also presented at the end of the current section.
\begin{theorem}\label{t:det}
    Let $G$ be a profinite group with an $\omega$\hyp{}presentation defined by an endomorphism $\phi$ such that $\det(M(\phi))\neq 0$. Then $G$ is a free profinite group if and only if $\phi$ is an automorphism.
\end{theorem}

The following example shows why the theorem may fail without the assumption that $\det(M(\phi))\neq 0$.
\begin{example}
    Let $A$ be an alphabet and $b$ a letter not in $A$. Consider the endomorphism $\phi$ of $\free(A\cup\{b\})$ defined by
    \begin{equation*}
        \phi(a) = 
        \begin{cases}
            a & \text{ if } a\neq b \\
            \emptyw & \text{ if } a=b.
        \end{cases}
    \end{equation*}
    It is straightforward to check that $\phi$ defines an $\omega$\hyp{}presentation of $\freepro(A)$, but it is clearly not an automorphism. 
\end{example}

We split the proof of Theorem~\ref{t:det} into two propositions. The first one relates freeness with minimal $\omega$\hyp{}presentations. The proof uses the fact that free profinite groups of finite ranks satisfy a topological version of the Hopfian property: every continuous surjective homomorphism between two free profinite groups of the same rank is an isomorphism \cite[Proposition~2.5.2]{Ribes2010a}. Also note the following straightforward consequence of the Hopfian property, which is used in the proof: the free profinite group over a finite set $A$ cannot be generated by strictly less than $\Card(A)$ elements, and therefore $d(\freepro(A)) = \Card(A)$.
\begin{proposition}\label{p:free}
    Let $\phi\in\End(\free(A))$ define a minimal $\omega$\hyp{}presentation of a profinite group $G$. Then $G$ is a free profinite group if and only if $\phi$ is an automorphism.
\end{proposition}

\begin{proof}
    Suppose that $\phi$ is an automorphism. Noting that the profinite completion is functorial \cite[Lemma~3.2.3]{Ribes2010a}, it follows that $\widehat{\phi}$ is also an automorphism, and by \cite[Proposition~3.7.4]{Almeida2020a}, $\widehat{\phi}^\omega$ is the identity. Since $\phi$ defines an $\omega$\hyp{}presentation of $G$, we see that
    \begin{equation*}
        G\isom \langle A \mid \widehat{\phi}^\omega(a)a^{-1} \given a\in A\rangle = \langle A \mid aa^{-1} \given a \in A \rangle = \langle A \mid \emptyw \rangle = \freepro(A).
    \end{equation*}

    Conversely, suppose that $G$ is a free profinite group. Since the $\omega$\hyp{}presentation of $G$ defined by $\phi$ is minimal, we have $d(G) = \Card(A)$ and it follows that $G$ is isomorphic to $\freepro(A)$. Moreover, by Lemma~\ref{l:kovacs}, $G$ is isomorphic to $\img(\widehat{\phi}^\omega)$. Therefore, $\widehat{\phi}^\omega\from\freepro(A)\to\img(\widehat{\phi}^\omega)$ is a continuous surjective homomorphism between free profinite groups of the same rank. By the Hopfian property, it follows that $\widehat{\phi}^\omega$ is injective. Since it is idempotent, we conclude that $\widehat{\phi}^\omega$ is the identity. By \cite[Proposition~3.7.4]{Almeida2020a}, $\widehat{\phi}$ is an automorphism and by \cite[Proposition~4.6.8]{Almeida2020a}, so is $\phi$.
\end{proof}

\begin{remark}
    At time of writing, we are not aware of any reliable way to find minimal $\omega$\hyp{}presentations for Schützenberger groups of primitive substitutions. Example~\ref{e:morse2} gives some clues as to why this might be a difficult problem.
\end{remark}

The second proposition, which completes the proof of Theorem~\ref{t:det}, gives a sufficient condition for an $\omega$\hyp{}presentation to be minimal. The minimal $\omega$\hyp{}presentation of $G(\tau)$ given at the end of Example~\ref{e:morse2} shows that this condition is not necessary.
\begin{proposition}\label{p:minimality}
    Let $\phi\in\End(\free(A))$ define an $\omega$\hyp{}presentation of a profinite group $G$ such that $\det(M(\phi))\neq 0$. Then the $\omega$\hyp{}presentation defined by $\phi$ is a minimal presentation of $G$.
\end{proposition}

The proof relies on a result from \cite{Almeida2013} which is recalled in the next proposition. Given $\psi\in\End(\freepro(A))$ and a finite group $H$, we define an operator $\psi_H\from H^A\to H^A$ as follows. A tuple $t\in H^A$, viewed as a map $A\to H$, extends uniquely to a continuous homomorphism $\widehat{t}\from\freepro(A)\to H$. We define $\psi_H(t)\in H^A$ by
\begin{equation*}
    \psi_H(t)(a) = \widehat{t}(\psi(a)), \qquad a\in A.
\end{equation*}
This construction gives a contravariant continuous action of the profinite monoid $\End(\freepro(A))$ on $H^A$ \cite[Lemma~3.1]{Almeida2013}.

Let $H$ be a finite group and $t\in H^A$ be a tuple. We say that $t$ \emph{generates $H$} if its components form a generating set of $H$.
\begin{proposition}[{\cite[Proposition~3.2]{Almeida2013}}]\label{p:tuple}
    Let $\phi\in\End(\free(A))$ define an $\omega$\hyp{}presentation of a profinite group $G$ and $H$ be a finite group. Then the following are equivalent:
    \begin{enumerate}
        \item $H$ is a continuous homomorphic image of $G$.
        \item There exist $t\in H^A$ and $n\geq 1$ such that $t$ generates $H$ and $\widehat{\phi}^n_H(t)=t$.
    \end{enumerate}
\end{proposition}

With this, we are ready for the proof of Proposition~\ref{p:minimality}, which also completes the proof of Theorem~\ref{t:det}.
\begin{proof}[Proof of Proposition~\ref{p:minimality}]
    Since every continuous homomorphic image $H$ of $G$ satisfies $d(H)\leq d(G)$, it suffices to show that one such image exists satisfying $d(H) = \Card(A)$. To this end, fix a prime $p$ that does not divide $\det(M(\phi))$ and let $H = (\zz/p\zz)^A$. Clearly, $d(H) = \Card(A)$. Let $M_p(\phi)$ be the reduction modulo $p$ of the incidence matrix $M(\phi)$. Then, a direct computation shows that for all $t\in H^A$ and $a\in A$,
    \begin{equation*}
        \widehat{\phi}_H(t)(a) = \sum_{b\in A}\epsilon_b(\phi(a))t(b) = (M_p(\phi)t)(a),
    \end{equation*}
    where $t$ is viewed as a column vector in the rightmost expression. By our choice of $p$, the determinant of $M_p(\phi)$ is an invertible element of $\zz/p\zz$, hence $M_p(\phi)$ is an invertible matrix over $\zz/p\zz$. Since invertible matrices of order $\Card(A)$ over $\zz/p\zz$ form a finite group, $M_p(\phi)^n$ is an identity matrix for some $n\geq 1$. It follows that 
    \begin{equation*}
        \widehat{\phi}^n_H(t) = M_p(\phi)^nt = t.
    \end{equation*}
    Hence, we may apply Proposition~\ref{p:tuple} with any tuple that generates $H$ (for instance, a tuple formed by a basis of $H$ as a vector space over $\zz/p\zz$), and we conclude that $H$ is a continuous homomorphic image of $G$.
\end{proof}

We finish this section by pointing out some interesting applications of Theorem~\ref{t:det}, starting with the following corollary:
\begin{corollary}\label{c:det}
    Let $G$ be a profinite group with an $\omega$\hyp{}presentation defined by an endomorphism $\phi$ such that $|\det(M(\phi))|>1$. Then $G$ is not a free profinite group.
\end{corollary}

\begin{proof}
    Suppose that $G$ is a free profinite group. By Theorem~\ref{t:det}, $\phi$ is an automorphism of $\free(A)$. But note that the incidence matrix defines a monoid homomorphism from $\End(\free(A))$ equipped with reversed composition, to the monoid of matrices of order $\Card(A)$ over $\zz$. In particular, it follows that the matrix $M(\phi)$ is invertible over $\zz$. Therefore, $|\det(M(\phi))|$ equals~1, a contradiction.
\end{proof}

Next, we present  two examples of primitive substitutions where the previous corollary may be used to show that the Schützenberger group is not free. The first example is due to Almeida, who proved that the Schützenberger group is non-free back in 2005 \cite[Example~7.2]{Almeida2007}.
\begin{example}[Almeida's example]\label{e:almeida}
    Let $\alpha$ be the primitive binary substitution defined by
    \begin{equation*}
        \begin{array}{rlll}
                \alpha\from & 0 & \mapsto & 01\\
                            & 1 & \mapsto & 0001.
        \end{array}
    \end{equation*}
    This substitution is aperiodic (using for instance \cite[Exercise~5.15]{Almeida2020a}) and proper. It follows that $\alpha$ defines an $\omega$\hyp{}presentation of $G(\alpha)$ (see Remark~\ref{r:proper}). A quick computation shows that $\det(M(\alpha))=-2$, hence $G(\alpha)$ is not a free profinite group by Corollary~\ref{c:det}. 
\end{example}

The second example is another well-known primitive substitution, although it appears as though its Schützenberger group has not been studied. 
\begin{example}
    The \emph{period doubling substitution} is the binary substitution $\varrho$ defined as follows:
    \begin{equation*}
        \begin{array}{rlll}
            \varrho\from & 0 & \mapsto & 01\\
                         & 1 & \mapsto & 00.
        \end{array}
    \end{equation*}
    It is a primitive substitution which is also aperiodic (using again \cite[Exercise~5.15]{Almeida2020a}). It admits $(1,0)$ as a connection of order 2. The return substitution $\retsubs{\varrho}{1,0}$ is given by
    \begin{equation*}
        \begin{array}{rlll}
            \retsubs{\varrho}{1,0}\from & 0 & \mapsto & 010\\
                                        & 1 & \mapsto & 01110.
        \end{array}
    \end{equation*}    
    The incidence matrix of $\retsubs{\varrho}{1,0}$ has determinant 4. By Theorem~\ref{t:return}, $\retsubs{\varrho}{1,0}$ defines an $\omega$-presentation of $G(\varrho)$, hence we may apply Corollary~\ref{c:det} to conclude that $G(\varrho)$ is not free.
\end{example}

\section{Schützenberger groups of relatively invertible substitutions}
\label{s:relfree}

In this section, we examine the Schützenberger groups of relatively invertible primitive substitutions, that is primitive substitutions that extend to automorphisms of some relatively free profinite group. To this end, it is useful to first recall a few basic things about pseudovarieties. A \emph{pseudovariety of groups}, or \emph{pseudovariety} for short, is a class $\var{H}$ of finite groups closed under taking subgroups, quotients and finite direct products. Here are a few common examples: 
\begin{itemize}
    \item the pseudovariety $\var{G}$ of all finite groups;
    \item the pseudovariety $\var{G}_p$ of finite $p$-groups, where $p$ is a given prime;
    \item the pseudovariety $\var{G}_{\nil}$ of finite nilpotent groups;
    \item the pseudovariety $\var{G}_{\sol}$ of finite solvable groups;
    \item the pseudovariety $\var{Ab}$ of finite Abelian groups.
\end{itemize}
A pseudovariety $\var{H}$ is called \emph{extension-closed} if for each $N, K\in\var{H}$, all extensions of $K$ by $N$ are in $\var{H}$. Among the examples given above, $\var{G}$, $\var{G}_p$ and $\var{G}_{\sol}$ are extension-closed, while $\var{G}_{\nil}$ and $\var{Ab}$ are not.

We denote by $\freepro_\var{H}(A)$ the free pro-$\var{H}$ group on a set $A$. As the name suggests, these are the free objects in the category of \emph{pro-$\var{H}$ groups} (residually $\var{H}$ compact groups). A detailed construction of free pro-$\var{H}$ groups can be found in \cite[Section~3]{Ribes2010a}. We call groups of the form $\freepro_\var{H}(A)$, where $\var{H}$ is a non-trivial pseudovariety, \emph{relatively free} profinite groups. For emphasis, we say that the groups $\freepro(A) = \freepro_\var{G}(A)$ are \emph{absolutely free}. It was shown in \cite[Theorems~7.2 and~7.6]{Almeida2013} that the Schützenberger groups of the substitutions $\tau$ and $\alpha$ presented in Examples~\ref{e:morse1} and~\ref{e:almeida} are not relatively free. In fact, at time of writing, there is no known example of a primitive substitution whose Schützenberger group is relatively free but not absolutely free. Part of our conclusion for this section, which is presented in Corollary~\ref{c:relfree}, states that the Schützenberger group of a primitive invertible substitution is absolutely free if and only if it is relatively free.

Let $\varphi\from A^*\to A^*$ be a primitive substitution and $\var{H}$ be a pseudovariety of groups. We denote by $\widehat{\varphi}_\var{H}$ the continuous endomorphism of $\freepro_\var{H}(A)$ naturally induced by $\varphi$. We say that $\varphi$ is \emph{$\var{H}$-invertible} if $\widehat{\varphi}_\var{H}$ is an automorphism, or equivalently if $\widehat{\varphi}_\var{H}^\omega$ is the identity \cite[Proposition~3.7.4]{Almeida2020a}. Note that $\var{G}$-invertibility is equivalent to invertibility in the usual sense. More explicitly, a primitive substitution extends to an automorphism of $\free(A)$ if and only if it extends to an automorphism of $\freepro(A)$ \cite[Proposition~4.6.8]{Almeida2020a}.

Primitive substitutions also determine pseudovarieties of their own, which have been introduced in \cite{Almeida2013}: let $\var{V}(\varphi)$ be the pseudovariety generated by the finite quotients of $G(\varphi)$, that is, by the finite groups that are continuous homomorphic images of $G(\varphi)$. Here is the main result of this section.
\begin{theorem}\label{t:relinv}
    Let $\var{H}$ be a non-trivial extension-closed pseudovariety and $\varphi$ be a primitive, aperiodic and $\var{H}$-invertible substitution. Then, $\var{H}$ is contained in $\var{V}(\varphi)$. 
\end{theorem}

The proof of this theorem relies on a number of intermediate results, starting with the technical lemma stated below. If $\var{H}$ is a non-trivial extension-closed pseudovariety, then free groups are residually $\var{H}$ \cite[Proposition~3.3.15]{Ribes2010a}, hence there is a natural embedding $\free(A)\hookrightarrow\freepro_\var{H}(A)$ for every alphabet $A$. The induced topology on $\free(A)$ is called the \emph{pro-$\var{H}$ topology}. If $X$ is a subset of $\free(A)$, then we denote its topological closure in $\freepro_\var{H}(A)$ by $\overline{X}_\var{H}$. On the other hand, we denote by $\cl_\var{H}(X)$ the closure of $X$ in the pro-$\var{H}$ topology of $\free(A)$. 

The proof of the next lemma is mostly a matter of combining several known results. We provide a proof for the sake of completeness. We chose to rely on \cite{Margolis2001,Ribes1994,Ribes2010a}, but let us mention that results from \cite{Coulbois2003} could be used as well. Alternatively, one could adapt the proof of \cite[Proposition~4.6.5]{Almeida2020a}, which can be partly traced back to \cite[Lemma~4.2]{Almeida2001}. 
\begin{lemma}\label{l:closure}
    Let $A$ be a finite set, $K$ be a finitely generated subgroup of $\free(A)$ and $\var{H}$ be a non-trivial extension-closed pseudovariety. Then, $\overline{K}_\var{H}$ is a free pro-$\var{H}$ group of rank at most that of $K$.
\end{lemma}

\begin{proof}
    By \cite[Proposition~3.4]{Ribes1994}, $\cl_\var{H}(K)$ is a subgroup of $\free(A)$ of rank at most that of $K$, so we may write $\cl_\var{H}(K)=\free(B)$, where $\Card(B)\leq\Card(A)$. Let $\iota\from\free(B)\hookrightarrow\free(A)$ be the inclusion, and denote by $\widehat{\iota}_\var{H}$ its extension to a continuous homormophism between the respective pro-$\var{H}$ completions. By \cite[Proposition~3.3.6]{Ribes2010a}, the pro-$\var{H}$ completion of a free group of finite rank is a free profinite group of the same rank, hence  we have $\widehat{\iota}_\var{H}\from\freepro_\var{H}(B)\to\freepro_\var{H}(A)$. Moreover, note that 
    \begin{equation*}
        \img(\widehat{\iota}_\var{H}) = \overline{\img(\iota)}_\var{H} = \overline{\cl_\var{H}(K)}_\var{H} = \overline{K}_\var{H},
    \end{equation*}
    where the leftmost equality follows from \cite[Lemma~3.2.4]{Ribes2010a}.  Therefore, it suffices to show that the pro-$\var{H}$ extension $\widehat{\iota}_\var{H}$ is injective. By \cite[Lemma~3.2.6]{Ribes2010a}, this is equivalent to showing that the pro-$\var{H}$ topology of $\cl_\var{H}(K)$ coincides with the subspace topology induced by the pro-$\var{H}$ topology of $\free(A)$. This last statement holds by the last part of \cite[Proposition~2.9]{Margolis2001}.
\end{proof}

We now turn to the following proposition, which is one of the main ingredients in the proof of Theorem~\ref{t:relinv}.
\begin{proposition}\label{p:freequotients}
    Let $\var{H}$ be a non-trivial extension-closed pseudovariety and $\varphi$ be a primitive, aperiodic and $\var{H}$-invertible substitution. Then, $G(\varphi)$ has a continuous homomorphic image isomorphic to a free pro-$\var{H}$ group of rank at least 2.
\end{proposition}

\begin{proof}
    Fix a connection $(u,v)$ of $\varphi$. By Theorem~\ref{t:return}, $\retsubs{\varphi}{u,v}$ defines an $\omega$\hyp{}presentation of $G(\varphi)$ and by Lemma~\ref{l:kovacs}, it follows that $G(\varphi)\isom\img(\widehat{\retsubs{\varphi}{u,v}}^\omega)$. Let $\widetilde{\loo{u,v}}$ be the natural extension of $\loo{u,v}$ to a continuous homomorphism $\widetilde{\loo{u,v}}\from\freepro(A_{u,v})\to \freepro_\var{H}(A)$. Recall that $\loo{u,v}\circ\retsubs{\varphi}{u,v} = \varphi^k\circ\loo{u,v}$, where $k$ is the order of the connection $(u,v)$, hence the following diagram is commutative:
    \begin{center}
        \begin{tikzcd}
            \freepro(A_{u,v}) \rar{\widetilde{\loo{u,v}}} \dar{\widehat{\retsubs{\varphi}{u,v}}^\omega} & \freepro_\var{H}(A) \dar{\widehat{\varphi}_\var{H}^\omega} \\
            \freepro(A_{u,v}) \rar{\widetilde{\loo{u,v}}} & \freepro_\var{H}(A).
        \end{tikzcd}.
    \end{center}
    Since $\varphi$ is $\var{H}$-invertible, $\widehat{\varphi}_{\var{H}}^\omega$ is the identity, hence $\img(\widetilde{\loo{u,v}})$ is a continuous homomorphic image of $G(\varphi)$. But notice that $\img(\widetilde{\loo{u,v}}) = \overline{\img(\loo{u,v})}_\var{H} = \overline{K}_{\var{H}}$, where $K$ is the subgroup of $\free(A)$ generated by the return set $\ret{u,v}$. By Lemma~\ref{l:closure}, $\overline{K}_{\var{H}}$ is free pro-$\var{H}$ group of finite rank. It remains only to show that this group has rank at least 2, or alternatively that this group is not commutative. But notice that $\overline{K}_{\var{H}}$ contains the submonoid of $A^*$ generated by $\ret{u,v}$, of which $\ret{u,v}$ itself forms a basis by \cite[Lemma~17]{Durand1999}. Since $\varphi$ is aperiodic, $\ret{u,v}$ must have at least 2 elements, and therefore it generates a non-commutative submonoid of $A^*$, thus concluding the proof.
\end{proof}

Next is another lemma, which all but completes the proof of our main result. This lemma is a consequence of an embedding result, due to Neumann and Neumann, dating back to 1959 \cite{Neumann1959}.
\begin{lemma}\label{l:2generated}
    Let $\var{H}$ be a non-trivial extension-closed pseudovariety. Then $\var{H}$ is generated, as a pseudovariety, by its 2-generated members.
\end{lemma}

\begin{proof}
    Let $L\in\var{H}$ be generated by non-identity elements $x_1,\dots, x_d$ of respective order $n_1,\dots, n_d$. The main construction of \cite{Neumann1959} implies that for all integers $m,n$ such that $m\geq 4d$ and $\lcm(n_1,\dots,n_d)\mid n$, we may embed $L$ in a 2-generated subgroup of the following wreath product:
    \begin{equation*}
        (L\wr \zz/n\zz)\wr\zz/m\zz.
    \end{equation*}
    Since $\var{H}$ is extension-closed, such a wreath product is in $\var{H}$ provided all the factors are in $\var{H}$. Therefore, it suffices to show that $m$ and $n$ can be chosen so that $\zz/m\zz, \zz/n\zz\in\var{H}$. For $n$, we may simply take $n=\lcm(n_1,\dots, n_d)$. Indeed, it then follows that $\zz/n\zz$ is a subgroup of $\zz/n_1\zz\times\dots\times\zz/n_d\zz$. This last group in turn lies in $\var{H}$ because, for $i=1,\dots,d$, the subgroup of $L$ generated by $x_i$ is isomorphic to $\zz/n_i\zz$. For $m$, choose some prime $p$ such that $\zz/p\zz\in\var{H}$, for instance a prime that divides one of the $n_i$. Since $\var{H}$ is extension-closed, it contains the extension-closed pseudovariety generated by $\zz/p\zz$, which is in fact $\var{G}_p$. In particular, $\var{H}$ contains $\zz/p^k\zz$ for all positive integers $k$. Taking $k\geq\log_p(4d)$, we find that $m=p^k$ fulfills all the required conditions.
\end{proof}

The proof of Theorem~\ref{t:relinv} is now a straightforward matter.
\begin{proof}[Proof of Theorem~\ref{t:relinv}]
    By Proposition~\ref{p:freequotients}, $\var{V}(\varphi)$ contains all 2-generated members of $\var{H}$. But by Lemma~\ref{l:2generated}, these groups generate $\var{H}$, hence $\var{H}\subseteq\var{V}(\varphi)$.
\end{proof}

Next, we proceed to highlight some consequences of our main result. A result of Almeida implies that a substitution is $\var{G}_p$-invertible if and only if $\det(M(\varphi))$ is not divisible by $p$ \cite[Proposition~5.2]{Almeida2002a}. Combining this with Theorem~\ref{t:relinv}, we immediately obtain the following: 
\begin{corollary}\label{c:primes}
    Let $\varphi$ be a primitive aperiodic substitution. Then $\var{G}_p$ is contained in $\var{V}(\varphi)$ for every prime $p$ that does not divide $\det(M(\varphi))$. In particular, if $\det(M(\varphi))$ is not 0, this must be the case for cofinitely many primes.
\end{corollary}

We now wish to show that Theorem~\ref{t:relinv} also holds for $\var{G}_{\nil}$, even though it is not extension-closed. By \cite[Corollary~5.3]{Almeida2002a}, a substitution $\varphi$ is $\var{G}_{\nil}$-invertible if and only if $\det(M(\varphi))=\pm1$. Substitutions satisfying the latter condition are called \emph{unimodular}. It turns out that for primitive substitutions, unimodularity implies aperiodicity. This is mostly thanks to a result of Holton and Zamboni from \cite{Holton1998}, as we now proceed to show.
\begin{proposition}\label{p:aperiodic}
    A primitive unimodular substitution $\varphi$ is aperiodic.
\end{proposition}

\begin{proof}
    First, note that the only unimodular substitution on a one-letter alphabet is the identity, which is not primitive. Thus, we may assume that $\varphi$ is defined on an alphabet with at least two letters. By a result of Holton and Zamboni, it suffices to show that the invertible matrix $M(\varphi)$ has an eigenvalue of modulus less than 1 \cite[Corollary~2.7]{Holton1998}. We argue by contradiction. Suppose that all eigenvalues of $M(\varphi)$ have modulus at least 1 and let $\rho$ be the spectral radius of $M(\varphi)$. Since $M(\varphi)$ is a primitive matrix, it follows from the Perron--Frobenius theorem that $\rho$ is a simple eigenvalue of $M(\varphi)$ with strictly maximal modulus \cite[Theorem~1.2.6]{Fogg2002}. Since $\varphi$ is defined on at least two letters, $M(\varphi)$ has one eigenvalue $\lambda$ distinct from $\rho$, and therefore $1\leq |\lambda|<\rho$. But $\det(M(\varphi))$ is the product of the eigenvalues of $M(\varphi)$, all of which have modulus at least 1, hence $|\det(M(\varphi))|\geq \rho > 1$. This contradicts the unimodularity of $\varphi$. 
\end{proof}

Tying up loose ends, we give a simple example showing that the conclusion of the previous proposition may not hold for substitutions that are $\var{G}_p$-invertible for cofinitely many primes.
\begin{example}
    Consider the following primitive substitution:
    \begin{equation*}
        \begin{array}{rlll}
            \varphi\from & 0 & \mapsto & 02\\
                         & 1 & \mapsto & 21\\
                         & 2 & \mapsto & 10.
        \end{array}
    \end{equation*}
    It is straightforward to check that $\det(M(\varphi))=-2$, hence $\varphi$ is $\var{G}_p$-invertible for all odd primes $p$. Yet, $\varphi$ is periodic, as the language of $\varphi$ consists of the factors in powers of the word $021$.
\end{example}

In the next corollary of Theorem~\ref{t:relinv}, we are able to omit the assumption of aperiodicity thanks to Proposition~\ref{p:aperiodic}.
\begin{corollary}\label{c:nil}
    If $\varphi$ is a primitive unimodular substitution, then $\var{G}_{\nil}$ is contained in $\var{V}(\varphi)$.
\end{corollary}

\begin{proof}
    Under our assumptions, $\varphi$ is $\var{G}_p$-invertible for all primes $p$, hence $\var{V}(\varphi)$ contains $\var{G}_p$ for all primes $p$ by Corollary~\ref{c:primes}. Since $\var{G}_{\nil}$ is the join (in the lattice of pseudovarieties ordered by inclusion) of the pseudovarieties $\var{G}_p$, where $p$ ranges over all primes, we find $\var{G}_{\nil}\subseteq\var{V}(\varphi)$.
\end{proof}

For the next corollary, which is just Theorem~\ref{t:relinv} with $\var{H}=\var{G}$, it is useful to note that invertible substitutions are unimodular. Indeed, recall from the proof of Corollary~\ref{c:det} that the incidence matrix of an automorphism must be invertible over $\zz$. In particular, we may again omit the aperiodicity assumption.
\begin{corollary}
    If $\varphi$ is a primitive invertible substitution, then $\var{V}(\varphi)$ equals $\var{G}$.
\end{corollary}

Finally, we give an application of Theorem~\ref{t:relinv} to the relative freeness question. For every pseudovariety $\var{H}$, all finite continuous homomorphic images of a pro-$\var{H}$ group lie in $\var{H}$ \cite[Theorem~2.1.3]{Ribes2010a}. In particular, if $G(\varphi)$ is a free pro-$\var{H}$ group, then $\var{V}(\varphi)\subseteq\var{H}$. Combining this observation with the three corollaries stated above yields the following corollary, which is our conclusion for this section.
\begin{corollary}\label{c:relfree}
    Let $\var{H}$ be a pseudovariety and $\varphi$ be a primitive substitution such that $G(\varphi)$ is a free pro-$\var{H}$ group.
    \begin{enumerate}
        \item If $\varphi$ is aperiodic, then $\var{G}_p\subseteq \var{H}$ for every prime $p$ that does not divide the determinant of $M(\varphi)$.
        \item If $\varphi$ is unimodular, then $\var{G}_{\nil}\subseteq\var{H}$.
        \item If $\varphi$ is invertible, then $\var{H}=\var{G}$ and therefore $G(\varphi)$ is absolutely free.
    \end{enumerate}
\end{corollary}

\section{An invertible substitution with a non-free Schützenberger group}
\label{s:invnonfree}

The aim of this section is to present a primitive invertible substitution whose Schützenberger group is not free, and thus not relatively free by Corollary~\ref{c:relfree}. This constitutes a counterexample to \cite[Corollary~5.7]{Almeida2007}. Let us formally state our conclusion. 
\begin{theorem}\label{t:counter}
    There exists an invertible primitive substitution whose Schützenberger group is not a relatively free profinite group.
\end{theorem}

Our example is the following substitution defined on $A_4=\{0,1,2,3\}$:
\begin{equation*}
    \begin{array}{rlll}
        \xi\from & 0 &\mapsto & 001\\
                 & 1 &\mapsto & 02\\
                 & 2 &\mapsto & 301\\
                 & 3 &\mapsto & 320.
    \end{array}
\end{equation*}

Showing that $\xi$ is primitive amounts to a straightforward computation. Moreover, one can show that $\xi$ is invertible by directly checking that
\begin{equation*}
    \begin{array}{rlll}
        \xi^{-1}\from & 0 &\mapsto & 1^{-1}02^{-1}3\\ 
                      & 1 &\mapsto & (3^{-1}20^{-1}1)^20\\ 
                      & 2 &\mapsto & 3^{-1}20^{-1}11\\ 
                      & 3 &\mapsto & 20^{-1}1^{-1}02^{-1}3.
    \end{array}
\end{equation*} 
Since invertible substitutions are unimodular, it follows from Proposition~\ref{p:aperiodic} that $\xi$ is aperiodic.

We proceed to show that $G(\xi)$ is not a free profinite group. In light of Theorem~\ref{t:det}, it suffices to show that $G(\xi)$ admits an $\omega$\hyp{}presentation defined by an endomorphism which is not an automorphism and whose incidence matrix has a non-zero determinant. This boils down to a series of computations organized as follows: 
\begin{enumerate}[label=\emph{Step \arabic*.},align=left]
    \item We compute the return substitution of $\xi$ with respect to the connection $(1,0)$. This defines an $\omega$\hyp{}presentation of $G(\xi)$ with 7 generators.  
    \item We compute the restriction $\rest{\retsubs{\xi}{1,0}}{1}$. This defines an $\omega$\hyp{}presentation of $G(\xi)$ with 5 generators, and moreover the incidence matrix of $\rest{\retsubs{\xi}{1,0}}{1}$ has non-zero determinant.
    \item We show that $\rest{\retsubs{\xi}{1,0}}{1}$ is not an automorphism of $\free(A_5)$.
\end{enumerate}

\subsection*{Step 1}

Let us compute the return substitution of $\xi$ with respect to the connection $(1,0)$. Note that this connection has order 2. For this computation, we use an algorithm described by Durand in \cite[p.5]{Durand2012}. A detailed implementation of Durand's algorithm written in pseudocode may be found in Algorithm~\ref{a:durand}. Given a primitive substitution $\varphi$ with a connection $(u,v)$ of order $k$, Durand's algorithm simultaneously computes the return substitution $\retsubs{\varphi}{u,v}$ and the bijection $\loo{u,v}\from A_{u,v}\to \ret{u,v}$ satisfying $\loo{u,v}\circ\retsubs{\varphi}{u,v} = \varphi^k\circ\loo{u,v}$. In particular, it can also be used to compute the return set.

\begin{algorithm}[t]
    \SetKwFunction{KwList}{list}
    \SetKwFunction{KwPop}{pop}
    \SetKwFunction{KwAppend}{append}
    \SetKw{KwLet}{let}
    \KwData{A primitive substitution $\varphi$ and a connection $(u,v)$ of $\varphi$ of order $k$.}
    \KwResult{The ordering $\loo{u,v}$ and the return substitution $\retsubs{\varphi}{u,v}$.}
    \Begin{
        $w\gets v$\;
        \Repeat{$uv$ occurs twice in $uw$}{
            $w\gets\varphi^k(w)$\;
        }
        \KwLet $\loo{u,v}(0)=\text{ leftmost return word in } uw$\;
        $i \gets 1$ \tcp*{least undefined letter of $\loo{u,v}$}
        $j \gets 0$ \tcp*{least undefined letter of $\retsubs{\varphi}{u,v}$}
        \While{j<i}{
            \ForEach{return word $r$ in $u\varphi^k(\loo{u,v}(j))v$}{
                \If{$r$ is not in $\img(\loo{u,v})$}{
                    \KwLet $\loo{u,v}(i) = r$\;
                    $i \gets i+1$\;
                }
            }
            \KwLet $\retsubs{\varphi}{u,v}(j) = \loo{u,v}^{-1}(\varphi^k(r))$\;
            $j\gets j+1$\;
        }
    }
    \caption{Durand's algorithm for computing return substitutions.}\label{a:durand}
\end{algorithm}

The first part of Algorithm~\ref{a:durand} (lines 1-6) computes the value of $\loo{u,v}(0)$. In the case at hand, we find that $001$ is the leftmost return word in $1\xi^{2}(0) = 100100102$, so $\loo{1,0}(0)=001$. Carrying out the rest of the algorithm yields the following result (see Table~\ref{tb:return} for details): 
\begin{equation*}
    \begin{array}{rlll}
        \loo{1,0}\from & 0 & \mapsto & 001\\
                          & 1 & \mapsto & 02001\\
                          & 2 & \mapsto & 02001301\\
                          & 3 & \mapsto & 02320001\\
                          & 4 & \mapsto & 02001301320301\\
                          & 5 & \mapsto & 02320301\\
                          & 6 & \mapsto & 001320001
    \end{array}\qquad
    \begin{array}{rlll}
        \retsubs{\xi}{1,0}\from & 0 & \mapsto & 00102\\
                       & 1 & \mapsto & 00310102\\
                       & 2 & \mapsto & 003101040002\\
                       & 3 & \mapsto & 003561010102\\
                       & 4 & \mapsto & 00310104000461050002\\
                       & 5 & \mapsto & 003561050002\\
                       & 6 & \mapsto & 0010461010102.
    \end{array}
\end{equation*}

\begin{table}\centering
    \begin{tabular}{>{\small}l|>{\small}l|>{\raggedright\arraybackslash\small}p{0.65\linewidth}}
        $\loo{1,0}^{-1}(r)$ & $r$ & $1\xi^2(r)0$ \\\hline\hline
        0 & $001$ & $1.\allowbreak001.\allowbreak001.\allowbreak02001.\allowbreak001.\allowbreak02001301.\allowbreak0$ \\\hline
        1 & $02001$ & $1.\allowbreak001.\allowbreak001.\allowbreak02320001.\allowbreak02001.\allowbreak001.\allowbreak02001.\allowbreak001.\allowbreak02001301.\allowbreak0$ \\\hline
        2 & $02001301$ & $1.\allowbreak001.\allowbreak001.\allowbreak02320001.\allowbreak02001.\allowbreak001.\allowbreak02001.\allowbreak001.\allowbreak02001301320301.\allowbreak001.\allowbreak001.\allowbreak001.\allowbreak02001301.\allowbreak0$ \\\hline
        3 & $02320001$ & $1.\allowbreak001.\allowbreak001.\allowbreak02320001.\allowbreak02320301.\allowbreak001320001.\allowbreak02001.\allowbreak001.\allowbreak02001.\allowbreak001.\allowbreak02001.\allowbreak001.\allowbreak02001301.\allowbreak0$ \\\hline
        4 & $02001301320301$ & $1.\allowbreak001.\allowbreak001.\allowbreak02320001.\allowbreak02001.\allowbreak001.\allowbreak02001.\allowbreak001.\allowbreak02001301320301.\allowbreak001.\allowbreak001.\allowbreak001.\allowbreak02001301320301.\allowbreak001320001.\allowbreak02001.\allowbreak001.\allowbreak02320301.\allowbreak001.\allowbreak001.\allowbreak001.\allowbreak02001301.\allowbreak0$ \\\hline
        5 & $02320301$ & $1.\allowbreak001.\allowbreak001.\allowbreak02320001.\allowbreak02320301.\allowbreak001320001.\allowbreak02001.\allowbreak001.\allowbreak02320301.\allowbreak001.\allowbreak001.\allowbreak001.\allowbreak02001301.\allowbreak0$ \\\hline
        6 & $001320001$ & $1.\allowbreak001.\allowbreak001.\allowbreak02001.\allowbreak001.\allowbreak02001301320301.\allowbreak001320001.\allowbreak02001.\allowbreak001.\allowbreak02001.\allowbreak001.\allowbreak02001.\allowbreak001.\allowbreak02001301.\allowbreak0$ 
    \end{tabular}
    \caption{Factorization of the words $1\xi^2(r)0$, $r\in\ret{1,0}$, used during the computation of the return substitution $\retsubs{\xi}{1,0}$.}\label{tb:return}
\end{table}

We recall that $\xi$ is primitive and unimodular (since it is invertible), hence it is aperiodic by Proposition~\ref{p:aperiodic}. Therefore, Theorem~\ref{t:return} shows that $\retsubs{\xi}{1,0}$ defines an $\omega$\hyp{}presentation of $G(\xi)$. This $\omega$\hyp{}presentation has 7 generators.

\subsection*{Step 2}

We now compute $\rest{\retsubs{\xi}{1,0}}{1}$, which we recall is the restriction of $\retsubs{\xi}{1,0}$ to the subgroup $\img(\retsubs{\xi}{1,0})$ of $\free(A_7)$. First, we need to find a basis of $\img(\retsubs{\xi}{1,0})$. To do this, it is convenient to recall some notions related with Stallings' algorithm. For a more exhaustive exposition of this topic, we point the reader to \cite{Kapovich2002}.

Let $\automaton{A}$ be a non-deterministic automaton over the alphabet $A$ with a distinguished state $s_0$, serving as both initial and final state. Let us also suppose that $\automaton{A}$ is weakly connected. We allow $\automaton{A}$ to also read words in $(A\cup A^{-1})^*$ in the natural way. More explicitly, if $a\in A$ acts partially on the states of $\automaton{A}$ by $x\mapsto x\cdot a$, then we let $a^{-1}$ act partially on the states of $\automaton{A}$ by
\begin{equation*}
    x\cdot a^{-1} = \{ y \given x\in y\cdot a\}.
\end{equation*}
We say that $\automaton{A}$ is \emph{folded} if no two distinct transitions exist that share the same label as well as the same origin or terminus. When $\automaton{A}$ is folded, it defines a subgroup $H_\automaton{A}$ of $\free(A)$ as follows: $x\in\free(A)$ belongs to $H_\automaton{A}$ if and only if the reduced word of $(A\cup A^{-1})^*$ representing $x$ is accepted by $\automaton{A}$ \cite[Lemma~3.2]{Kapovich2002}. Furthermore, we can obtain a basis for the subgroup $H_\automaton{A}$ as follows. Let $T$ be a spanning tree of $\automaton{A}$. Given two states $x,y\in\automaton{A}$, we denote by $[x,y]_T$ the unique path between $x$ and $y$ in $T$. Let $T'$ be the set of transitions of $\automaton{A}$ that do not belong to $T$. For each $e\in T'$, let $b_e$ be the label of the path $[s_0,x]_T e [y,s_0]_T$, where $x$ and $y$ are respectively the origin and terminus of $e$. Then, the set $X_T = \{ b_e \given e \in T'\}$ is a basis of $H_\automaton{A}$ \cite[Lemma~6.1]{Kapovich2002}.

Let us use this to obtain a basis of $\img(\retsubs{\xi}{1,0})$. First, note the two following equalities, which can be checked with direct computations:
\begin{equation*}
    \retsubs{\xi}{1,0}(6)=\retsubs{\xi}{1,0}(02^{-1}45^{-1}3),\quad \retsubs{\xi}{1,0}(4)=\retsubs{\xi}{1,0}(21^{-1}25^{-1}31^{-1}5).
\end{equation*}
It follows that $\img(\retsubs{\xi}{1,0})$ is generated by $\retsubs{\xi}{1,0}(B)$, where $B = \{0,1,2,3,5\}$. Let
\begin{equation*}
    Y = \{ 0,\ 10^{-1},\ 1^{-1}2,\ 1^{-1}25^{-1}31^{-1}30^{-1},\ 03^{-1}52^{-1}1\}.
\end{equation*} 
It is not hard to see that $Y$ generates $\free(B)$ (it is even a basis of $\free(B)$ since $Y$ and $B$ have the same number of elements). Therefore, $\img(\retsubs{\xi}{1,0})$ is generated by the set
\begin{equation*}
    \retsubs{\xi}{1,0}(Y) = \{00102,\ 00310^{-1},\ 2^{-1}40002,\ 2^{-1}461010^{-1},\ 01^{-1}54^{-1}2\}.
\end{equation*}

\begin{figure}\centering
    \begin{tikzpicture}[scale=.45]
        \node (s0) at (20bp,120bp) [draw,circle, double] {$s_0$};
        \node (s1) at (100bp,40bp) [draw,circle] {$s_1$};
        \node (s2) at (100bp,200bp) [draw,circle] {$s_2$};
        \node (s3) at (200bp,40bp) [draw,circle] {$s_3$};
        \node (s4) at (300bp,40bp) [draw,circle] {$s_4$};
        \node (s5) at (300bp,200bp) [draw,circle] {$s_5$};
        \node (s6) at (100bp,120bp) [draw,circle] {$s_6$};
        \node (s7) at (400bp,40bp) [draw,circle] {$s_7$};
        \node (s8) at (200bp,120bp) [draw,circle] {$s_8$};
        \node (s9) at (400bp,120bp) [draw,circle] {$s_9$};
        \draw [->,dashed] (s0) -- node {$0$} (s1);
        \draw [<-,dashed] (s0) -- node {$2$} (s2);
        \draw [->,dashed] (s1) -- node {$0$} (s3);
        \draw [<-,dashed] (s1) to[bend right] node {$1$} (s4);
        \draw [->,dashed] (s2) -- node {$4$} (s5);
        \draw [<-,dashed] (s2) -- node {$0$} (s6);
        \draw [->] (s3) -- node {$3$} (s4);
        \draw [->] (s3) to node {$1$} (s6);
        \draw [->] (s4) -- node {$5$} (s5);
        \draw [<-,dashed] (s4) -- node {$0$} (s7);
        \draw [->,dashed] (s5) -- node {$0$} (s8);
        \draw [->,dashed] (s5) -- node {$6$} (s9);
        \draw [<-] (s6) -- node {$0$} (s8);
        \draw [<-] (s7) -- node {$1$} (s9);
    \end{tikzpicture}
    \caption{An automaton over the alphabet $A_7$. The distinguished state is identified by a double circle and a spanning tree is highlighted with dashed edges.}
    \label{f:automaton-return}
\end{figure}
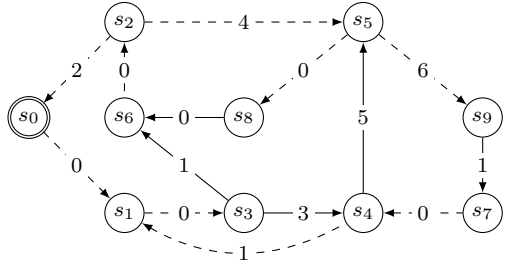

Let $X = \retsubs{\xi}{1,0}(Y)$. We claim that $X$ is a basis of $\img(\retsubs{\xi}{1,0})$. Indeed, consider the automaton $\automaton{A}$ over the alphabet $A_7$ presented in Figure~\ref{f:automaton-return}, where a spanning tree $T$ is highlighted. A direct verification reveals that $\automaton{A}$ is folded and that $X = X_T$, so $X$ is a basis of $\img(\retsubs{\xi}{1,0})$ by \cite[Lemma~6.1]{Kapovich2002}. The restriction $\rest{\retsubs{\xi}{1,0}}{1}$, written in the basis $X$ ordered as above, is (see Table~\ref{tb:restriction} for details): 
\begin{equation*}
    \begin{array}{rlll}
        \rest{\retsubs{\xi}{1,0}}{1}\from & 0 & \mapsto & 00100102\\
                          & 1 & \mapsto & 0014301\\
                          & 2 & \mapsto & 342000102\\
                          & 3 & \mapsto & 3420301001\\
                          & 4 & \mapsto & 4.
    \end{array}
\end{equation*}

\begin{table}
    \begin{tabular}{>{\small}l|>{\raggedright\arraybackslash\small}p{0.7\linewidth}}
        $b$ & $\retsubs{\xi}{1,0}(b)$ \\\hline\hline
		$00102$ & $00102.\allowbreak00102.\allowbreak00310^{-1}.\allowbreak00102.\allowbreak00102.\allowbreak00310^{-1}.\allowbreak00102.\allowbreak2^{-1}40002$ \\\hline
		$00310^{-1}$ & $00102.\allowbreak00102.\allowbreak00310^{-1}.\allowbreak01^{-1}54^{-1}2.\allowbreak2^{-1}461010^{-1}.\allowbreak00102.\allowbreak00310^{-1}$ \\\hline
		$2^{-1}40002$ & $2^{-1}461010^{-1}.\allowbreak01^{-1}54^{-1}2.\allowbreak2^{-1}40002.\allowbreak00102.\allowbreak00102.\allowbreak00102.\allowbreak00310^{-1}.\allowbreak00102.\allowbreak2^{-1}40002$ \\\hline
		$2^{-1}461010^{-1}$ & $2^{-1}461010^{-1}.\allowbreak01^{-1}54^{-1}2.\allowbreak2^{-1}40002.\allowbreak00102.\allowbreak2^{-1}461010^{-1}.\allowbreak00102.\allowbreak00310^{-1}.\allowbreak00102.\allowbreak00102.\allowbreak00310^{-1}$ \\\hline
		$01^{-1}54^{-1}2$ & $01^{-1}54^{-1}2$ 
    \end{tabular}
    \caption{Factorization of the elements $\retsubs{\xi}{1,0}(b)$, $b\in X$, used to compute the restriction $\rest{\retsubs{\xi}{1,0}}{1}$.}\label{tb:restriction}
\end{table}
By Proposition~\ref{p:restriction}, we conclude that $\rest{\retsubs{\xi}{1,0}}{1}$ defines an $\omega$\hyp{}presentation of $G(\xi)$. The incidence matrix of $\rest{\retsubs{\xi}{1,0}}{1}$, which has determinant 1, is given by 
\begin{equation*}
    M(\rest{\retsubs{\xi}{1,0}}{1}) = 
    \begin{pmatrix}
        5 & 2 & 1 & 0 & 0 \\
        3 & 2 & 0 & 1 & 1 \\
        4 & 1 & 2 & 1 & 1 \\
        4 & 2 & 1 & 2 & 1 \\
        0 & 0 & 0 & 0 & 1 \\
    \end{pmatrix}.
\end{equation*}

\subsection*{Step 3}

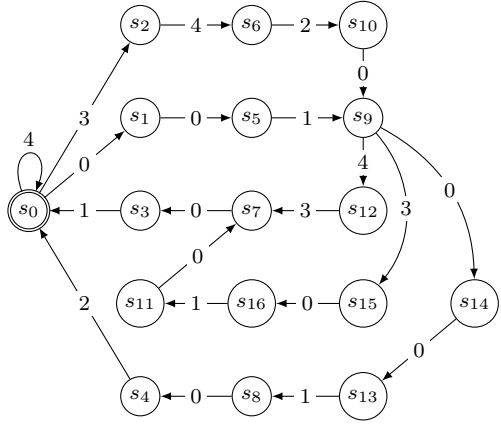
\begin{figure}\centering
    \begin{tikzpicture}[scale=.35]
        \node (s0) at (20bp,100bp) [draw,circle, double] {$s_0$};
        \node (s1) at (140bp,200bp) [draw,circle] {$s_1$};
        \node (s2) at (140bp,300bp) [draw,circle] {$s_2$};
        \node (s3) at (140bp,100bp) [draw,circle] {$s_3$};
        \node (s4) at (140bp,-100bp) [draw,circle] {$s_4$};
        \node (s5) at (260bp,200bp) [draw,circle] {$s_5$};
        \node (s6) at (260bp,300bp) [draw,circle] {$s_6$};
        \node (s7) at (260bp,100bp) [draw,circle] {$s_7$};
        \node (s8) at (260bp,-100bp) [draw,circle] {$s_8$};
        \node (s9) at (380bp,200bp) [draw,circle] {$s_9$};
        \node (s10) at (380bp,300bp) [draw,circle] {$s_{10}$};
        \node (s11) at (140bp,0bp) [draw,circle] {$s_{11}$};
        \node (s12) at (380bp,100bp) [draw,circle] {$s_{12}$};
        \node (s13) at (380bp,-100bp) [draw,circle] {$s_{13}$};
        \node (s14) at (500bp,0bp) [draw,circle] {$s_{14}$};
        \node (s15) at (380bp,0bp) [draw,circle] {$s_{15}$};
        \node (s16) at (260bp,0bp) [draw,circle] {$s_{16}$};
        \draw [->] (s0) to[loop] node[above=1pt] {$4$} (s0);
        \draw [->] (s0) -- node {$0$} (s1);
        \draw [->] (s0) -- node {$3$} (s2);
        \draw [<-] (s0) -- node {$1$} (s3);
        \draw [<-] (s0) -- node {$2$} (s4);
        \draw [->] (s1) -- node {$0$} (s5);
        \draw [->] (s2) -- node {$4$} (s6);
        \draw [<-] (s3) -- node {$0$} (s7);
        \draw [<-] (s4) -- node {$0$} (s8);
        \draw [->] (s5) -- node {$1$} (s9);
        \draw [->] (s6) -- node {$2$} (s10);
        \draw [<-] (s7) to node {$0$} (s11);
        \draw [<-] (s7) -- node {$3$} (s12);
        \draw [<-] (s8) -- node {$1$} (s13);
        \draw [<-] (s9) -- node {$0$} (s10);
        \draw [->] (s9) -- node {$4$} (s12);
        \draw [->] (s9) to[bend left] node {$0$} (s14);
        \draw [->] (s9) to[bend left=40] node {$3$} (s15);
        \draw [<-] (s11) -- node {$1$} (s16);
        \draw [<-] (s13) -- node {$0$} (s14);
        \draw [->] (s15) -- node {$0$} (s16);
    \end{tikzpicture}
    \caption{An automaton over the alphabet $A_5$. The distinguished state is identified by a double circle.}
    \label{f:automaton-restriction}
\end{figure}

To conclude the proof of Theorem~\ref{t:counter}, it remains only to show that $\rest{\retsubs{\xi}{1,0}}{1}$ is not an automorphism of $\free(A_5)$. Consider the automaton $\automaton{A}$ over the alphabet $A_5$ presented in Figure~\ref{f:automaton-restriction}. A simple inspection of each of its 17 states shows that $\automaton{A}$ is folded, hence it defines a proper subgroup $H_\automaton{A}$ of $\free(A_5)$. Moreover, the words $\rest{\retsubs{\xi}{1,0}}{1}(a)$ for $a\in A_5$ are all accepted by $\automaton{A}$, hence $\img(\rest{\retsubs{\xi}{1,0}}{1})\leq H_{\automaton{A}}$. Therefore, $\img(\rest{\retsubs{\xi}{1,0}}{1})$ is also a proper subgroup of $\free(A_5)$ and $\rest{\retsubs{\xi}{1,0}}{1}$ is not an automorphism of $\free(A_5)$.

\section*{Acknowledgements}

I am indebted to Alfredo Costa for many helpful comments which greatly simplified Section~\ref{s:invnonfree} and overall made this paper much clearer. I also wish to thank Jorge Almeida for insightful discussions and for numerous suggestions, some of which led to more general results in Section~\ref{s:relfree}.


\begin{thebibliography}{10}

\bibitem{Almeida2001}
J.~Almeida.
\newblock Dynamics of implicit operations and tameness of pseudovarieties of
  groups.
\newblock {\em Trans. Amer. Math. Soc.}, 354(1):387--411, 2001.

\bibitem{Almeida2002a}
J.~Almeida.
\newblock Dynamics of finite semigroups.
\newblock In {\em Semigroups, Algorithms, Automata and Languages ({C}oimbra,
  2001)}, pages 269--292. World Scientific, 2002.

\bibitem{Almeida2005}
J.~Almeida.
\newblock Profinite semigroups and applications.
\newblock In V.~Kudryavtsev and I.~G. Rosenberg, editors, {\em Structural
  {Theory} of {Automata}, {Semigroups}, and {Universal} {Algebra}}, volume 207
  of {\em {NATO} {Science} {Series} {II}: {Mathematics}, {Physics} and
  {Chemistry}}, pages 1--45, Dordrecht, 2005. Springer.

\bibitem{Almeida2007}
J.~Almeida.
\newblock Profinite groups associated with weakly primitive substitutions.
\newblock {\em J. Math. Sci.}, 144(2):3881--3903, 2007.
\newblock Translated from Fundam. Prikl. Mat., Vol. 11, No. 3, pp. 13--48,
  2005.

\bibitem{Almeida2013}
J.~Almeida and A.~Costa.
\newblock Presentations of {S}ch\"utzenberger groups of minimal subshifts.
\newblock {\em Israel J. Math.}, 196(1):1--31, 2013.

\bibitem{Almeida2016}
J.~Almeida and A.~Costa.
\newblock A geometric interpretation of the {S}ch\"utzenberger group of a
  minimal subshift.
\newblock {\em Ark. Mat.}, 54(2):243--275, 2016.

\bibitem{Almeida2020a}
J.~Almeida, A.~Costa, R.~Kyriakoglou, and D.~Perrin.
\newblock {\em Profinite Semigroups {a}nd Symbolic Dynamics}.
\newblock Springer International Publishing, 2020.

\bibitem{Berthe2015}
V.~Berth{\'{e}}, C.~De Felice, F.~Dolce, J.~Leroy, D.~Perrin, C.~Reutenauer,
  and G.~Rindone.
\newblock Acyclic, connected and tree sets.
\newblock {\em Monatsh. Math.}, 176(4):521--550, 2015.

\bibitem{Berthe2015b}
V.~Berth{\'{e}}, C.~De Felice, F.~Dolce, J.~Leroy, D.~Perrin, C.~Reutenauer,
  and G.~Rindone.
\newblock Bifix codes and interval exchanges.
\newblock {\em J. Pure Appl. Algebra}, 219(7):2781--2798, 2015.

\bibitem{Berthe2015a}
V.~Berth{\'{e}}, C.~De Felice, F.~Dolce, J.~Leroy, D.~Perrin, C.~Reutenauer,
  and G.~Rindone.
\newblock Maximal bifix decoding.
\newblock {\em Discrete Math.}, 338(5):725--742, 2015.

\bibitem{Costa2006}
A.~Costa.
\newblock Conjugacy invariants of subshifts: An approach from profinite
  semigroup theory.
\newblock {\em Int. J. Algebra Comput.}, 16(4):629--655, 2006.

\bibitem{Costa2011}
A.~Costa and B.~Steinberg.
\newblock Profinite groups associated to sofic shifts are free.
\newblock {\em Proc. London Math. Soc.}, 102(3):341--369, 2011.

\bibitem{Coulbois2003}
T.~Coulbois, M.~Sapir, and P.~Weil.
\newblock A note on the continuous extensions of injective morphisms between
  free groups to relatively free profinite groups.
\newblock {\em Publ. Mat.}, 47(2):477--487, 2003.

\bibitem{Durand1998}
F.~Durand.
\newblock A characterization of substitutive sequences using return words.
\newblock {\em Discrete Math.}, 179(1-3):89--101, 1998.

\bibitem{Durand2012}
F.~Durand.
\newblock {HD0L}-{$\omega$}-equivalence and periodicity problems in the
  primitive case (to the memory of {G}. {R}auzy).
\newblock {\em Unif. Distrib. Theory}, 7(1):199--215, 2012.

\bibitem{Durand1999}
F.~Durand, B.~Host, and C.~Skau.
\newblock Substitutional dynamical systems, {B}ratteli diagrams and dimension
  groups.
\newblock {\em Ergodic Theory Dynam. Systems}, 19(4):953--993, 1999.

\bibitem{Fogg2002}
N.~P. Fogg.
\newblock {\em Substitutions in Dynamics, Arithmetics and Combinatorics}.
\newblock Springer Berlin Heidelberg, 2002.

\bibitem{Holton1998}
C.~Holton and L.~Q. Zamboni.
\newblock Geometric realizations of substitutions.
\newblock {\em Bull. Soc. Math. Fr.}, 126(2):149--179, 1998.

\bibitem{Hunter1983}
R.~P. Hunter.
\newblock Some remarks on subgroups defined by the {Bohr} compactification.
\newblock {\em Semigr. Forum}, 26(1):125--137, 1983.

\bibitem{Kapovich2002}
I.~Kapovich and A.~Myasnikov.
\newblock {S}tallings foldings and subgroups of free groups.
\newblock {\em J. Algebra}, 248(2):608--668, 2002.

\bibitem{Lubotzky2001}
A.~Lubotzky.
\newblock Pro-finite presentations.
\newblock {\em J. Algebra}, 242(2):672--690, 2001.

\bibitem{Margolis2001}
S.~Margolis, M.~Sapir, and P.~Weil.
\newblock Closed subgroups in pro-{V} topologies and the extension problem for
  inverse automata.
\newblock {\em Int. J. Algebra Comput.}, 11(04):405--445, 2001.

\bibitem{Neumann1959}
B.~H. Neumann and H.~Neumann.
\newblock Embedding theorems for groups.
\newblock {\em J. London Math. Soc.}, s1-34(4):465--479, 1959.

\bibitem{Neumann1967}
H.~Neumann.
\newblock {\em Varieties of Groups}.
\newblock Springer Berlin Heidelberg, 1967.

\bibitem{Ribes1994}
L.~Ribes and P.~A. Zaleskii.
\newblock The pro-p topology of a free groups and algorithmic problems in
  semigroups.
\newblock {\em Int. J. Algebra Comput.}, 04(03):359--374, 1994.

\bibitem{Ribes2010a}
L.~Ribes and P.~Zalesskii.
\newblock {\em Profinite Groups}.
\newblock Springer Berlin Heidelberg, second edition, 2010.

\bibitem{Steinberg2010}
B.~Steinberg.
\newblock On the endomorphism monoid of a profinite semigroup.
\newblock {\em Port. Math.}, 68:177--183, 2010.

\end{thebibliography}
\end{document}